\documentclass[preprint,12pt,3p,times]{elsarticle}

\usepackage{amsthm,amssymb,amsmath,amsfonts} %%%amsthm,amsmath,amsfonts
\usepackage{eucal}      %mathcal
\usepackage{mathrsfs}   %mathscr
\usepackage{longtable}
\usepackage{multirow}
\usepackage{verbatim}
\usepackage{placeins}
\usepackage{graphicx}

\def\Rset{\mathbb{R}}
\def\Cset{\mathbb{C}}
\def\Nset{\mathbb{N}}
\def\Kset{\mathbb{K}}

\theoremstyle{plain}
\newtheorem{thm}{Theorem}[section]
\newtheorem{lem}[thm]{Lemma}
\newtheorem{cor}[thm]{Corollary}

\theoremstyle{definition}
\newtheorem{defn}[thm]{Definition}
\newtheorem{rem}[thm]{Remark}
\newtheorem{exmp}[thm]{Example}

\newtheorem*{correctness21}{\bf Proof of the correctness of Definition~2.1}
\newtheorem*{correctness24}{\bf Proof of the correctness of Definition~2.4}

\usepackage{caption}
\numberwithin{equation}{section}

\journal{}

\begin{document}

\begin{frontmatter}
\title{On a family of Weierstrass-type root-finding methods with accelerated convergence}
\author{Petko D. Proinov}\ead{proinov@uni-plovdiv.bg}
\author{Maria T. Vasileva}\ead{mariavasileva@uni-plovdiv.bg}
\address{Faculty of Mathematics and Informatics, University of Plovdiv, Plovdiv 4000, Bulgaria}

\begin{abstract}
Kyurkchiev and Andreev (1985) constructed an infinite sequence of Weierstrass-type iterative methods for approximating all zeros of a polynomial simultaneously. The first member of this sequence of iterative methods is the famous method of Weierstrass (1891) 
and the second one is the method of Nourein (1977). For a given integer $N \ge 1$, the $N$th method of this family has the order of convergence ${N+1}$. Currently in the literature, there are only local convergence results for these methods.
The main purpose of this paper is to present semilocal convergence results for the Weierstrass-type methods under computationally verifiable initial conditions and with computationally verifiable a posteriori error estimates.
\end{abstract}


\end{frontmatter}

%\tableofcontents

%%%%%%%%%%%%%%%%%%%%%%%%%%%%%%%%%%%%%%%%%%%%%%%%%%
%%
%%        Introduction and preliminaries
%%
%%%%%%%%%%%%%%%%%%%%%%%%%%%%%%%%%%%%%%%%%%%%%%%%%%

%Section 1
\section{Introduction and preliminaries}
\label{sec:introduction}

Throughout this paper ${(\Kset,|\cdot|)}$ denotes an algebraically closed normed field and
$\Kset[z]$ denotes the ring of polynomials (in one variable) over $\Kset$. 
We endow the vector space $\Kset^n$ with the $p$-norm
${\|x\|_p = \left( \sum _{i = 1} ^n |x_i|^p \right) ^{1/p}}$ for some ${1 \le p \le \infty}$,
and we equip ${(\Rset^n,\|\cdot\|_p)}$ with coordinate-wise ordering ${\preceq}$ defined by
\begin{equation} \label{eq:coordinate-wise-ordering}
x \preceq y  \quad\text{if and only if}\quad x_i \le y_i \,\, \text{ for each } \,\, i = 1, \ldots, n.
\end{equation}
Then ${(\Rset^n,\|\cdot\|_p)}$ is a solid vector space.
Also we define a cone norm ${\|\cdot\|}$ in ${\Kset^n}$ with values in ${\Rset^n}$ by 
\begin{equation} \label{eq:cone-norm}
{\|x\| = (|x_1|,\ldots,|x_n|)}.
\end{equation}
Then ${(\Kset^n,\|\cdot\|)}$ is a cone normed space over $\Rset^n$ (see, e.g., Proinov \cite{Pro13}).

Let ${f \in \Kset[z]}$ be a polynomial of degree ${n \ge 2}$.
A vector ${\xi \in \Kset^n}$ is called a root-vector of $f$ if 
\(
{f(z) = a_0 \prod _{i = 1}^ n (z - \xi_i)}
\)
for all ${z \in \Kset}$, where ${a_0 \in \Kset}$.

In 1891, Weierstrass \cite{Wei91} published his famous iterative method for simultaneous computation of all zeros of $f$.
The \emph{Weierstrass method} is defined by the following iteration
\begin{equation} \label{eq:Weierstrass-iteration}
x^{(k + 1)}  = x^{(k)} - W_f(x^{(k)}), \qquad k = 0, 1, 2,\ldots,
\end{equation}
where the operator ${W_f \colon \mathcal{D} \subset \Kset^n \to \Kset^n}$ is defined by
\begin{equation} \label{eq:Weierstrass-correction}
W_f(x) = (W_1(x),\ldots,W_1(x)) \quad\text{with}\quad
W_i(x) = \frac{f(x_i)}{ a_0 \prod_{j \ne i} {(x_i  - x_j )}} \quad (i = 1,\ldots,n),
\end{equation}
where ${a_0 \in \Kset}$ is the leading coefficient of $f$ and the domain $\mathcal{D}$ of $W_f$ is the set of all vectors in $\Kset^n$ with distinct components.
The Weierstrass method \eqref{eq:Weierstrass-iteration} has second-order of convergence provided that all zeros of $f$ are simple.
Other iterative methods for simultaneous finding polynomial zeros can be found in the books \cite{Kyu98,McN07,Pet08,SAK94} and the references therein.

In 1985, Kyurkchiev and Andreev \cite{KA85} introduced a sequence of iterative methods for approximating all zeros of a polynomial simultaneously. The first member of their family of iterative methods is the Weierstrass method \eqref{eq:Weierstrass-iteration} 
and the second one is the method of Nourein \cite{Nou77}. 

Before we present Kyurkchiev and Andreev's family of iterative methods, 
we give some notations which will be used throughout the paper.
We define the binary relation $\#$ on $\Kset^n$ by
\begin{equation} \label{eq:special-binary-relation}
	x \, \# \, y  \quad\Leftrightarrow\quad  x_i \ne y_j \, \text{ for all } \,  i,j \in I_n \, \text{ with } \,  i \ne j.
\end{equation}
Here and throughout this paper, we denote by $I_n$ the set of indices ${1, \ldots, n}$. 
For two vectors ${x \in \Kset^n}$ and ${y \in \Rset^n}$ we define in ${\Rset^n}$ the vector
\[
\frac{x}{y} = \left( \frac{|x_1|}{y_1},\cdots,\frac{|x_n|}{y_n} \right)
\]
provided that $y$ has only nonzero components. 
We define the function ${d \colon \Kset^n \to \Rset^n}$ by %%%%${d(x) = (d_1(x),\ldots,d_n(x))}$ with
\[
d(x) = (d_1(x),\ldots,d_n(x)) \quad\text{with}\quad d_i(x) = \min_{j \ne i} |x_i - x_j| \quad (i = 1,\ldots,n).
\]
In the sequel, for a given vector $x$ in $\Kset^n$, ${x_i}$ always denotes the $i$th component of $x$. 
In particular, if $F$ is a map with values in $\Kset^n$, then ${F_i(x)}$ denotes the $i$th component of the vector $F(x)$.

%Definition 1.1
\begin{defn} \label{def:iteration-function-Weierstrass-type}
Suppose ${f \in \Kset[z]}$ is a polynomial of degree ${n \ge 2}$.
We define the sequence ${(T^{(N)})_{N = 0}^\infty}$ of functions ${T^{(N)} \colon D_N \subset \Kset^n \to \Kset^n}$
recursively by setting ${T^{(0)}(x) = x}$ and
\begin{equation} \label{eq:Weierstrass-high-order-iteration-function}
T_i^{(N + 1)} (x) = x_i -  \frac{f(x_i)}{\displaystyle \prod_{j \ne i} (x_i - T_j^{(N)}(x))}
\quad (i=1,\ldots,n),
\end{equation}
where the sequence of the domains ${D_N}$ is also defined recursively by setting ${D_0 = \Kset^n}$ and
\begin{equation} \label{eq:domain-Weierstrass-high-order}
    D_{N + 1} =  \{ x \in D_N \colon  x \, \# \,T^{(N)}(x) \}.
\end{equation}
\end{defn}

Let ${N \in \Nset}$ be fixed. Then the $N$th method of Kyurkchiev-Andreev's family can be defined by the following fixed point iteration
\begin{equation} \label{eq:Weierstrass-high-order}
x^{(k + 1)} = T^{(N)} (x^{(k)}), \qquad k = 0,1,2,\ldots.
\end{equation}

Currently in the literature, there are only local convergence results for the Weierstrass-type methods 
\eqref{eq:Weierstrass-high-order} (see \cite{KA85,PV15}).
In this paper, we present semilocal convergence results for the Weierstrass-type methods under computationally verifiable initial conditions and with computationally verifiable a posteriori error estimates. 
These results are obtained by using some results of \cite{Pro15a} and \cite{Pro15b}.

The paper is structured as follows: In Section~\ref{sec:Local-convergence-theorem-of-the second type}, we obtain new local convergence results (Theorem~\ref{thm:second-local-Weierstrass}, Corollary~\ref{cor:second-local-Weierstrass-type} and 
Corollary~\ref{cor:second-local-Weierstrass}) for the Weierstrass-type methods \eqref{eq:Weierstrass-high-order}.
In the case ${N = 1}$ (Weierstrass method) and ${p =\infty}$ the main result of this section reduces 
to a result of Proinov \cite[Theorem~7.3]{Pro15a}. 
In Section~\ref{sec:semilocal-theorem-Weierstrass}, we present our semilocal convergence results 
(Theorem~\ref{thm:semilocal-theorem-Weierstrass-N}, Theorem~\ref{thm:semilocal-theorem-Weierstrass},
Corollary~\ref{cor:semilocal-theorem-Weierstrass-first} and Corollary~\ref{cor:semilocal-theorem-Weierstrass-second}) 
for the Weierstrass-type methods 
\eqref{eq:Weierstrass-high-order}.
Note that these results are based on the local convergence results obtained in the previous section. 
In Section~\ref{sec:numerical-examples-Weierstrass}, we provide three numerical examples to show the applicability of our semilocal convergence results.

Throughout this paper, we follow the terminology of \cite{Pro15a}.
In particular we refer to this paper for the definition of the following notions:
quasi-homogeneous function of degree $r \ge 0$; gauge function of order ${r \ge 1}$;
function of initial conditions of a map; initial point of a map; iterated contraction at a point.

%%%%%%%%%%%%%%%%%%%%%%%%%%%%%%%%%%%%%%%%%%%%%%%%%%%%%%%%%%%%%%%%%%%%%%
%%
%%    Local convergence analysis of the Weierstrass-type methods
%%
%%%%%%%%%%%%%%%%%%%%%%%%%%%%%%%%%%%%%%%%%%%%%%%%%%%%%%%%%%%%%%%%%%%%%%

%Section 2
\section{Local convergence analysis of the Weierstrass-type methods}
\label{sec:Local-convergence-theorem-of-the second type}

Let ${f \in \Kset[z]}$ be a polynomial of degree $n \ge 2$.
In this section, we study the convergence of the Weierstrass-type methods \eqref{eq:Weierstrass-high-order} with respect to the  function of initial conditions
${E \colon \mathcal{D} \to \Rset_+}$  defined by
\begin{equation} \label{eq:FIC-second-kind}
E(x) = \left \|\frac{x - \xi}{d(x)} \right \|_p
\end{equation}
for some ${1 \le p \le \infty}$.
We define the function ${\Psi \colon \Rset_+ \to \Rset_+}$ by
\begin{equation} \label{eq:Psi-Weierstrass}
\Psi(t) = (1 + 2 t) \left( 1 + \frac{t}{(n - 1)^p} \right)^{n - 1}.
\end{equation}
Throughout this section we denote by $R$ the unique positive solution of the equation ${\Psi(t) = 2}$. It can be proved that
\begin{equation} \label{eq:R-second-local-inequality}
\frac{n (2^{1/n} - 1)}{(n - 1)^{1/q } + 2} < R < \frac{1}{2} \, .
\end{equation}
where $q$ is the conjugate exponent of $p$, i.e. $q$ is defined by means of ${1 \le q \le \infty}$ and ${1/p + 1/q = 1}$.
The lower estimate in \eqref{eq:R-second-local-inequality} can be proved as
Lemma~7.4 of \cite{Pro15a} and the upper estimate is trivial.
Also, we define the functions ${\omega \colon \Rset_+ \to \Rset_+}$ and ${\phi \colon [0,R] \to [0,1]}$ by
\begin{equation} \label{eq:omega-phi-Weierstrass}
\omega(t) = \left( 1 + \frac{t}{(n - 1)^p} \right)^{n - 1}
\quad\text{and}\quad
\phi(t) = \frac{\omega(t) - 1}{1 - 2 t \omega(t)} \, .
\end{equation}
It follows from the definition of $R$ that
\begin{equation} \label{eq:omega-R}
\omega(R) = 2 / (1 + 2 R), \quad  \omega(R) < 2, \quad \omega(R) < 1 / (2 R) \quad\text{and}\quad \phi(R) =1.
\end{equation}

%Definition 2.1
\begin{defn} \label{def:phi-N-Weierstrass-second}
We define the sequence $(\phi_N)_{N=0}^\infty$ of nondecreasing functions ${\phi_N \colon [0,R] \to [0,1]}$ recursively by setting
${\phi_{0}(t) = 1}$ and
\begin{equation} \label{eq:phi-N-Weierstrass-second}
 \phi_{N + 1}(t) = \frac{\omega_N(t) - 1}{1 - 2 t \omega_N(t)}, \quad\text{where}\quad
 \omega_N(t) = \left( 1 + \frac{t \phi_N(t)}{(n - 1)^p} \right)^{n - 1} \, .
\end{equation}
\end{defn}

\begin{correctness21}
We prove the correctness of the definition by induction.
For $N = 0$ it is obvious.
Assume that for some ${N \ge 0}$ the function $\phi_N$ is well-defined and nondecreasing on ${[0,R]}$ and ${\phi_N(R) = 1}$.
We shall prove the same for $\phi_{N + 1}$.
From the induction hypothesis, we get that $\omega_N$ is nondecreasing on ${[0,R]}$ and ${\omega_N(R) = \omega(R)}$.
Then it follows from \eqref{eq:omega-R} that
 \begin{equation} \label{eq:denominator-of-phi-N-positive-Weierstrass}
1 - 2 t \omega_N(t) \ge 1 - 2 R \omega_N(R) = 1 - 2 R \omega(R) > 0
\end{equation}
which guarantees that the function ${\phi_{N + 1}}$ is well-defined on ${[0, R]}$.
Obviously, ${\phi_{N + 1}}$ is nondecreasing on ${[0, R]}$.
From the definition of $\phi_{N + 1}$, ${\omega_N(R) = \omega(R)}$  and \eqref{eq:omega-R}, we get
\[
 \phi_{N + 1}(R) = \frac{\omega_N(R) - 1}{1 - 2 R \omega_N(R)} = \frac{\omega(R) - 1}{1 - 2 R \omega(R)}= \phi(R) = 1.
\]
This completes the induction and the proof of the correctness of Definition~\ref{def:phi-N-Weierstrass-second}.
\qed
\end{correctness21}

%Definition 2.2
\begin{defn} \label{def:varphi-N-Weierstrass-second}
Given ${N \ge 0}$, we define the function ${\varphi_N \colon [0,R] \to [0,R]}$ by
 \begin{equation} \label{eq:varphi-N-Weierstrass}
 \varphi_N(t) = t \phi_{N}(t).
\end{equation}
\end{defn}

%Lemma 2.3
\begin{lem} \label{lem:phi-N-properties-Weierstrass-second}
Let ${N \ge 0}$. Then:
\begin{enumerate}[(i)]
  \item $\phi_N$ is a quasi-homogeneous of degree ${N}$ on ${[0,R]}$;
  \item $\phi_{N + 1}(t) \le \phi(t) \phi_N(t)$ for every ${t \in [0,R]}$;
  \item ${\phi_N(t) \le \phi(t)^N}$ for every ${t \in [0,R]}$;
  \item $\varphi_N$ is a gauge function of order ${N + 1}$ on ${[0, R]}$.
\end{enumerate}
\end{lem}

\begin{proof}
We prove Claim (i) by induction. The case ${N = 0}$ is obvious.
From the induction hypothesis and Example 2.2 of \cite{Pro15a}, we conclude that the function ${\omega_N(t) - 1}$ is a quasi-homogeneous of degree ${N + 1}$ on ${[0,R]}$. Hence, $\phi_{N + 1}$ is quasi-homogeneous of degree ${N + 1}$ on ${[0, R]}$ as a product of a quasi-homogeneous of degree ${N + 1}$ and a nondecreasing function. This ends the proof of (i).
From the fact that the function ${\omega_N(t) - 1}$ is a quasi-homogeneous of degree ${N + 1}$ on ${[0,R]}$, we obtain
\[
\omega_N(t) - 1 \le \phi_N(t) (\omega(t) - 1).
\]
From \eqref{eq:phi-N-Weierstrass-second}, the last inequality and ${\omega_N(t) \le \omega(t)}$, we get
\[
\phi_{N + 1}(t) = \frac{\omega_N(t) - 1}{1 - 2 t \omega_N(t)} \le \phi_{N}(t) \frac{\omega(t) - 1}{1 - 2 t \omega(t)} =
\phi_{N}(t) \phi(t)
\]
which proves (ii). Claim (iii) follows from (ii) by induction.
Claim (iv) follows from (i) and definition of $\varphi_N$.
\end{proof}

%Definition 2.4
\begin{defn} \label{def:beta-psi-N-Weierstrass}
For a given integer ${N \ge 1}$, we define the increasing function ${\beta_N \colon [0,R] \to [0,1)}$ by
\begin{equation}\label{eq:beta-N-Weierstrass}
   \beta_{N} (t) = \omega_{N - 1} (t) - 1
\end{equation}
and we define the decreasing function ${\psi_N \colon [0,R] \to (0,1]}$ by
\begin{equation}\label{eq:psi-N-Weierstrass}
   \psi_{N} (t) = 1 - 2 t \omega_{N - 1} (t),
\end{equation}
where the function $\omega_N$ is defined in \eqref{eq:phi-N-Weierstrass-second}.
\end{defn}

\begin{correctness24}
The functions $\beta_N$ and $\psi_N$ are well defined on ${[0,R]}$ since $\omega_{N-1}$ is well defined on this interval.
The monotonicity of these functions is obvious. It remains to prove that ${\beta_N(R) < 1}$ and ${\psi_N(R) > 0}$.
It follows from ${\omega_{N-1}(R) = \omega(R)}$ and \eqref{eq:omega-R} that
\[
\beta_N(R)= \omega(R) -1 < 1 \quad\text{and}\quad \psi_N(R)= 1 - 2 R \omega(R) > 0
\]
which completes the proof of the correctness of Definition~\ref{def:phi-N-Weierstrass-second}.
\qed
\end{correctness24}

%Lemma 2.5
\begin{lem} \label{lem:beta-psi-N-properties-Weierstrass}
Let ${N \ge 1}$. Then
\begin{enumerate}[(i)]
  \item $\beta_N$ is a quasi-homogeneous of degree ${N}$ on ${[0,R]}$;
  \item ${\beta_N(t) = \phi_N(t) \psi_N(t)}$ for every ${t \in [0,R]}$.
\end{enumerate}
\end{lem}

\begin{proof}
Claim (i) follows from Lemma~\ref{lem:phi-N-properties-Weierstrass-second}(i) and Example 2.2 of \cite{Pro15a}.
Claim (ii) follows from Definition~\ref{def:phi-N-Weierstrass-second} and
Definition~\ref{def:beta-psi-N-Weierstrass}.
\end{proof}

%Lemma 2.6
\begin{lem} [\cite{PV15}] \label{lem:Weierstrass-high-order-map}
Let ${f \in \Kset[z]}$ be a polynomial of degree ${n \ge 2}$. Assume that ${\xi \in \Kset^n}$ is a root-vector of $f$ and ${N \ge 1}$.
If ${x \in D_{N + 1}}$, then for every ${i \in I_n}$
\begin{equation} \label{eq:Weierstrass-high-order-map}
T^{(N + 1)}_i(x) - \xi_i = \left(1 - \prod_{j \ne i} (1 + u_j)  \right) (x_i - \xi_i),
\end{equation}
where ${u_j \in \Kset}$ is defined by
\begin{equation} \label{eq:u-definition}
u_j = \frac{T_j ^{(N)}(x) - \xi_j}{x_i - T_j ^{(N)}(x)} \, .
\end{equation}
\end{lem}

%Lemma 2.7
\begin{lem} [\cite{PV15b}]  \label{lem:u-v-inequalities-2}
Let ${u,v,\xi \in \Kset^n}$, ${\alpha \ge 0}$ and ${1 \le p \le \infty}$. If $v$ is a vector with distinct components such that
\begin{equation} \label{eq:inequality-vectors-cone-norm-h-condition}
    \|u - \xi\| \preceq \alpha \|v - \xi\|,
\end{equation}
then for all ${i,j \in I_n}$,
\begin{equation} \label{eq:u-v-inequalities-2}
    |u_j - v_i| \ge \left( 1 - (1 + \alpha) \left\| \frac{v - \xi}{d(v)} \right\|_p \right) |v_i - v_j|.
\end{equation}
\end{lem}

%Lemma 2.8
\begin{lem} [\cite{PV15b}] \label{lem:u-v-inequalities-3}
Let ${u,v,\xi \in \Kset^n}$, ${\alpha \ge 0}$ and ${1 \le p \le \infty}$. If $v$ is a vector with distinct components such that \eqref{eq:inequality-vectors-cone-norm-h-condition} holds,
then for all ${i,j \in I_n}$,
\begin{equation} \label{eq:u-v-inequalities-3}
    |u_i - u_j| \ge \left(1 - 2^{1/q} \, (1 + \alpha) \left\| \frac{v - \xi}{d(v)} \right\|_p  \right) |v_i - v_j|.
\end{equation}
\end{lem}

%Lemma 2.9
\begin{lem} \label{lem:Weierstrass-iterated-contraction-inequalities-second-FIC}
Let ${f \in \Kset[z]}$ be a polynomial of degree ${n \ge 2}$, ${\xi \in \Kset^n}$ be a root-vector of $f$, ${N \ge 1}$ and
${1 \le p \le \infty}$. Suppose ${x \in \Kset^n}$ is a
vector with distinct components such that
\begin{equation} \label{eq:Weierstrass-second-FIC-condition}
E(x) \le R,
\end{equation}
where the function $E$ is defined by \eqref{eq:FIC-second-kind}.
Then $f$ has only simple zeros in $\Kset$, ${x \in D_N}$ and
\begin{equation} \label{eq:Weierstrass-contraction-map-inequality-second-FIC}
\|T^{(N)} (x) - \xi \| \preceq \beta_N(E(x)) \|x - \xi \|.
 \end{equation}
\end{lem}

\begin{proof}
It follows from Proposition 5.3 of \cite{Pro15a} that the vector $\xi$ has distinct components,
which means that $f$ has only simple zeros in $\Kset$.
Further, we proceed by induction.
If ${N = 1}$, then the proof can be found in \cite{Pro15a}.
Assume that both ${x \in D_N}$ and \eqref{eq:Weierstrass-contraction-map-inequality-second-FIC} hold for some ${N \ge 1}$.

First, we prove that ${x \in D_{N + 1}}$.
It follows from \eqref{eq:Weierstrass-contraction-map-inequality-second-FIC} that condition
\eqref{eq:inequality-vectors-cone-norm-h-condition}
is satisfied with ${u = T^{(N)}(x)}$, ${v = x}$ and ${\alpha = 1}$.
Therefore, by Lemma~\ref{lem:u-v-inequalities-2}, taking into account \eqref{eq:Weierstrass-second-FIC-condition} and the fact that $x$ is a vector with distinct components, we obtain
\begin{equation} \label{eq:x-TN-inequality-second-Weierstrass}
     |x_i - T_j^{(N)}(x)| \ge \left( 1 - 2 \left\| \frac{x - \xi}{d(x)} \right\|_p \right) |x_i - x_j|
     \ge (1 -  2 E(x)) \, d_j(x) > 0
\end{equation}
for every ${j \ne i}$.
Consequently, ${x \, \# \, T^{(N)}(x)}$ which proves that ${x \in D_{N + 1}}$.

Second, we shall prove that \eqref{eq:Weierstrass-contraction-map-inequality-second-FIC} is true for ${N + 1}$.
Obviously, the last statement is equivalent to
\begin{equation} \label{eq:Weierstrass-contraction-map-inequality-i-coordinate-second}
 |T^{(N + 1)}_i(x) - \xi_i| \le  \beta_{N + 1}(E(x)) |x_i - \xi_i|  \quad\text{for every } i \in I_n \, .
\end{equation}
Let ${i \in I_n}$ be fixed. We consider the vector ${u = (u_j)_{j \ne i} \in \Kset^{n - 1}}$, where ${u_j}$ is defined by \eqref{eq:u-definition}.
It follows from \eqref{eq:Weierstrass-contraction-map-inequality-second-FIC} and \eqref{eq:x-TN-inequality-second-Weierstrass} that
\[
|u_j| =  \frac{|T_j ^{(N)}(x) - \xi_j|}{|x_i - T_j ^{(N)}(x)|} \le \frac{\beta_N(E(x))}{1 - 2 E(x)} \frac{|x_j - \xi_j|}{d_j(x)} \, .
\]
Taking the $p$-norm and using Lemma~\ref{lem:beta-psi-N-properties-Weierstrass}(ii) and ${\psi_N(t) \le 1}$, we get
\begin{equation} \label{eq:u-p-norm-inequality-second_FIC}
\|u\|_p \le \frac{E(x) \beta_N(E(x))}{1 - 2 E(x)}
= \frac{E(x) \phi_N(E(x)) \psi_N(E(x))}{1 - 2 E(x)}
\le E(x) \phi_N(E(x)).
\end{equation}
Then, from Lemma~\ref{lem:Weierstrass-high-order-map}, Lemma~3.2 of \cite{PC14},
\eqref{eq:u-p-norm-inequality-second_FIC} and \eqref{eq:beta-N-Weierstrass}, we obtain
\begin{eqnarray*}
 |T^{(N + 1)}_i(x) - \xi_i| & = & \left|\prod_{j \ne i} (1 + u_j) - 1 \right| |x_i - \xi_i| \\
& \le & \left[ \left(1 + \frac{\|u\|_p}{(n - 1)^{1/p}} \right)^{n - 1} - 1 \right] |x_i - \xi_i|\\
& = & (\omega_{N}(E(x)) - 1) |x_i - \xi_i| = \beta_{N + 1}(E(x)) |x_i - \xi_i|
\end{eqnarray*}
which proves that \eqref{eq:Weierstrass-contraction-map-inequality-i-coordinate-second} is true for ${N + 1}$.
This completes the induction and the proof.
\end{proof}

%Lemma 2.10
\begin{lem} \label{lem:Weierstrass-iterated-contraction-second-FIC}
Let ${f \in \Kset[z]}$ be a polynomial of degree ${n \ge 2}$, ${\xi \in \Kset^n}$ be a root-vector of $f$, ${N \ge 1}$ and
${1 \le p \le \infty}$.
Let ${T^{(N)} \colon D_N \subset \Kset^n \to \Kset^n}$  and
${E \colon D_N \to \Rset_+}$ be defined by Definition~\ref{def:iteration-function-Weierstrass-type} and
\eqref{eq:FIC-second-kind}, respectively. Then:
\begin{enumerate}[(i)]
  \item  ${E}$ is a function of initial conditions of $T^{(N)}$ with a gauge function $\varphi_N$ of order ${N + 1}$ on ${J = [0,R]}$;
  \item ${T^{(N)}}$ is an iterated contraction at $\xi$ with respect to $E$ with control function ${\beta_N}$;
  \item Every point ${x^{(0)} \in \Kset^n}$ such that ${E(x^{(0)}) \in J}$  is an initial point of ${T^{(N)}}$.
\end{enumerate}
\end{lem}

\begin{proof}
(i)
First, we prove that
\begin{equation} \label{eq:Weierstrass-FIN-inequality-second-FIC}
E(T^{(N)} (x)) \le \varphi_N(E(x)) \quad\text{for all } \, x \in \mathcal{D} \, \text{ such that } \, E(x) \in J.
\end{equation}
Condition \eqref{eq:Weierstrass-contraction-map-inequality-second-FIC} allow us to apply Lemma~\ref{lem:u-v-inequalities-3} with
${u = T^{(N)}(x) }$, ${v = x}$ and ${\alpha = \beta_N(E(x))}$.
Therefore, we get
\[
 |T_i^{(N)}(x) - T_j^{(N)}(x)| \ge (1 - 2^{1/q} E(x)(1 + \beta_N(E(x))) |x_i - x_j| \ge 
%(1 - 2 E(x)(1 + \beta_N(E(x))) |x_i - x_j| =
 \psi_N(E(x)) |x_i - x_j|.
\]
Taking minimum over ${j \ne i}$ on both sides of this inequality, we obtain
\begin{equation} \label{eq:d-Tx-inequality-psi-Weierstrass}
 d_i(T^{(N)}(x)) \ge \psi_N(E(x)) d_i(x) > 0.
\end{equation}
It follows from \eqref{eq:Weierstrass-contraction-map-inequality-i-coordinate-second},
\eqref{eq:d-Tx-inequality-psi-Weierstrass} and Lemma~\ref{lem:beta-psi-N-properties-Weierstrass}(ii) that
\[
   \frac{|T_i^{(N)}(x) - \xi_i|}{ d_i(T^{(N)}(x)) }
   \le \frac{\beta_N(E(x))}{\psi_N(E(x))} \frac{|x_i - \xi_i|}{d_i(x)}
   = \phi_{N}(E(x)) \frac{|x_i - \xi_i|}{d_i(x)} \, .
\]
Taking the $p$-norm on both sides of this inequality, we get
\begin{equation} \label{eq:Weierstrass-FIN-inequality-second-FIC-2}
   E(T^{(N)} (x)) \le \phi_{N}(E(x))\, E(x) = \varphi_N(E(x))
\end{equation}
which proves \eqref{eq:Weierstrass-FIN-inequality-second-FIC}.
Now Claim (i) follows from \eqref{eq:Weierstrass-FIN-inequality-second-FIC} and Lemma~\ref{lem:phi-N-properties-Weierstrass-second}(iv).

(ii)
It follows from Lemma~\ref{lem:Weierstrass-iterated-contraction-inequalities-second-FIC}.

(iii)
It follows from Lemma~\ref{lem:Weierstrass-iterated-contraction-inequalities-second-FIC} that ${x^{(0)} \in D_N}$.
According to Proposition~2.7 of \cite{Pro15a} to prove that ${x^{(0)}}$ is an initial point of ${T^{(N)}}$
it is sufficient to prove that
\begin{equation} \label{eq:initial-point-test-Weierstrass}
	x \in D_N \,\, \text{ and } \,\, E(x) \in J \,\, \Rightarrow \,\, T^{(N)}(x) \in D_N \, .
\end{equation}
From ${x \in D_N}$, we conclude that ${T^{(N)}(x) \in \Kset^n}$.
It follows from \eqref{eq:d-Tx-inequality-psi-Weierstrass} that ${T^{(N)}(x) \in \mathcal{D}}$.
The inequality \eqref{eq:Weierstrass-FIN-inequality-second-FIC-2} implies ${E(T^{(N)}(x)) \in J}$ since ${\varphi_N \colon J \to J}$ and
${E(x) \in J}$.
Thus we have both ${T^{(N)}(x) \in \mathcal{D}}$ and ${E(T^{(N)}(x)) \in J}$.
Applying Lemma~\ref{lem:Weierstrass-iterated-contraction-inequalities-second-FIC} to the vector ${T^{(N)}(x)}$,
we get ${T^{(N)}(x) \in D_N}$ which proves \eqref{eq:initial-point-test-Weierstrass}.
Hence, ${x^{(0)}}$ is an initial point of ${T^{(N)}}$.
\end{proof}

Now, we are ready to state and prove the main result in this section. In the case ${N = 1}$ and ${p = \infty}$ this result reduces to
Theorem~7.3 of \cite{Pro15a}.

%Theorem 2.11
\begin{thm} \label{thm:second-local-Weierstrass}
Let ${f \in \Kset[z]}$ be a polynomial of degree ${n \ge 2}$,
${\xi \in \Kset^n}$ be a root-vector of $f$, ${N \ge 1}$ and ${1 \le p \le \infty}$.
Suppose ${x^{(0)} \in \Kset^n}$ is a vector with distinct components such that
\begin{equation} \label{eq:second-local-Weierstrass-initial-condition}
\Psi(E(x^{(0)})) \le 2,
\end{equation}
where the function $E$ is defined by \eqref{eq:FIC-second-kind} and $\Psi$ is defined by \eqref{eq:Psi-Weierstrass}.
Then $f$ has only simple zeros and the Weierstrass-type iteration \eqref{eq:Weierstrass-high-order} is well-defined and converges to $\xi$ with error estimates
\begin{equation} \label{eq:second-local-Weierstrass-error-estimate}
\|x^{(k+1)} - \xi\| \preceq \theta \lambda^{(N + 1)^k} \, \|x^{(k)} - \xi\|
\quad\text{and}\quad
  \|x^{(k)} - \xi\| \preceq \theta^ k \lambda^{((N + 1)^k - 1) /N} \|x^{(0)} - \xi\|,
\end{equation}
for all ${k \ge 0}$, where ${\lambda = \phi_N(E(x^{(0)}))}$, ${\theta = \psi_N(E(x^{(0)}))}$ and the real functions
$\phi_N$ and $\psi_N$ are defined by Definition~\ref{def:phi-N-Weierstrass-second} and Definition~\ref{def:beta-psi-N-Weierstrass}, respectively.
Moreover, if the inequality in \eqref{eq:second-local-Weierstrass-initial-condition} is strict, then
the Weierstrass-type iteration converges to $\xi$ with order of convergence ${N + 1}$.
\end{thm}

\begin{proof}
Since the function $\Psi$ is increasing on $\Rset_+$ and $R$ is the unique positive solution of the equation ${\Psi(t) = 2}$,
then the initial condition \eqref{eq:second-local-Weierstrass-initial-condition} is equivalent to ${E(x^{(0)}) \in J}$,
 where ${J = [0,R]}$.
Now the statement of Theorem~\ref{thm:second-local-Weierstrass} follows from Theorem~3.3 of \cite{Pro15a},
Lemma~\ref{lem:Weierstrass-iterated-contraction-second-FIC} and Lemma~\ref{lem:beta-psi-N-properties-Weierstrass}.
\end{proof}

The following result is a simplified version of Theorem~\ref{thm:second-local-Weierstrass}.
It involves only the function $\phi$ defined by \eqref{eq:omega-phi-Weierstrass}.

%Corollary 2.12
\begin{cor} \label{cor:second-local-Weierstrass-type}
Let ${f \in \Kset[z]}$ be a polynomial of degree ${n \ge 2}$,
${\xi \in \Kset^n}$ be a root-vector of $f$, ${N \ge 1}$ and ${1 \le p \le \infty}$.
Suppose ${x^{(0)} \in \Kset^n}$ is a vector with distinct components satisfying
\eqref{eq:second-local-Weierstrass-initial-condition}.
Then $f$ has only simple zeros and the Weierstrass-type iteration \eqref{eq:Weierstrass-high-order} is well-defined and converges to $\xi$ with error estimates
\begin{equation} \label{eq:second-local-Weierstrass-cor-error-estimate}
\|x^{(k+1)} - \xi\| \preceq \lambda^{N (N + 1)^k}  \, \|x^{(k)} - \xi\|
\quad\text{and}\quad
  \|x^{(k)} - \xi\| \preceq \lambda^{(N + 1)^k - 1} \|x^{(0)} - \xi\|
\end{equation}
for all ${k \ge 0}$, where ${\lambda = \phi(E(x^{(0)}))}$ and the real function
$\phi$ is defined by \eqref{eq:omega-phi-Weierstrass}.
Moreover, if the inequality in \eqref{eq:second-local-Weierstrass-initial-condition} is strict, then
the Weierstrass-type iteration converges to $\xi$ with order of convergence ${N + 1}$.
\end{cor}

\begin{proof}
It follows from Theorem~\ref{thm:second-local-Weierstrass}, Lemma~\ref{lem:phi-N-properties-Weierstrass-second}(ii)
and the inequality ${0 < \psi_N(t) \le 1}$ which holds for every ${t \in [0, R]}$.
\end{proof}

We end this section with a convergence theorem under initial condition in explicit form.

%Corollary 2.13
\begin{cor} \label{cor:second-local-Weierstrass}
Let ${f \in \Kset[z]}$ be a polynomial of degree ${n \ge 2}$,
${\xi \in \Kset^n}$ be a root-vector of $f$, ${N \ge 1}$ and ${1 \le p \le \infty}$.
Suppose ${x^{(0)} \in \Kset^n}$ is a vector with distinct components such that
\begin{equation} \label{eq:second-local-Weierstrass-initial-condition-cor}
E(x^{(0)}) = \left \| \frac{x^{(0)} - \xi}{d(x^{(0)})} \right \|_p \le  \frac{n (2^{1/n} - 1)}{(n - 1)^{1/q} + 2} \, ,
\end{equation}
where the function $E$ is defined by \eqref{eq:FIC-second-kind} and $q$ is the conjugate exponent of $p$.
Then the Weierstrass-type iteration \eqref{eq:Weierstrass-high-order} is well-defined and converges to $\xi$ with order of convergence
${N + 1}$ and with error estimates
\eqref{eq:second-local-Weierstrass-error-estimate} and \eqref{eq:second-local-Weierstrass-cor-error-estimate}.
\end{cor}

\begin{proof}
It follows from Theorem~\ref{thm:second-local-Weierstrass}, Corollary~\ref{cor:second-local-Weierstrass-type}
and the lower estimate in \eqref{eq:R-second-local-inequality}.
\end{proof}

%%%%%%%%%%%%%%%%%%%%%%%%%%%%%%%%%%%%%%%%%%%%%%%%%%%%%%%%
%%
%%        Semilocal convergence theorems for the Weierstrass-type methods
%%
%%%%%%%%%%%%%%%%%%%%%%%%%%%%%%%%%%%%%%%%%%%%%%%%%%%%%%%%

%Section 3
\section{Semilocal convergence theorems for the Weierstrass-type methods}
\label{sec:semilocal-theorem-Weierstrass}

In this section, we present two semilocal convergence theorem for Weierstrass-type methods \eqref{eq:Weierstrass-high-order}
with computationally verifiable initial conditions and with computationally verifiable a posteriori error estimates.
We study the convergence of the Weierstrass-type methods \eqref{eq:Weierstrass-high-order} with respect to
the  function of initial conditions ${E \colon \mathcal{D} \to \Rset_+}$ defined by
\begin{equation} \label{eq:FIC-third}
E_f(x) =  \left \| \frac{W_f(x)}{d(x)} \right \|_p \qquad (1 \le p \le \infty).
\end{equation}

\subsection{First semilocal convergence theorem}
\label{subsec:First-semilocal-theorem-Weierstrass}

Recently Proinov \cite{Pro15b} have proposed a new approach for obtaining semilocal convergence results for simultaneous methods via local convergence results. In particular, from Theorem~\ref{thm:second-local-Weierstrass}, we can obtain a convergence theorem for
Weierstrass-type methods \eqref{eq:Weierstrass-high-order} under computationally verifiable initial conditions.
In what follows $q$ is the conjugate exponent of $p$.

%Theorem 3.1
\begin{thm}[Proinov \cite{Pro15b}] \label{prop:semilocal-theorem-second-type-Weierstrass}
Let  ${f \in \Kset[z]}$ be
a polynomial of degree ${n \ge 2}$. Suppose there exists a vector ${x \in \Kset^n}$ with
distinct components such that
\begin{equation} \label{eq:semilocal-theorem-initial-condition-second-kind-Weierstrass}
 E_f(x) = \left \| \frac{W_f(x)}{d(x)} \right \|_p \le \frac{1}{(1 + \sqrt{a})^2}
\end{equation}
for some ${1 \le p \le \infty}$, where ${a = (n - 1) ^ {1/q}}$.
In the case ${n = 2}$ and ${p = \infty}$ we assume that inequality in
\eqref{eq:semilocal-theorem-initial-condition-second-kind-Weierstrass} is strict.
Then $f$ has only simple zeros and there exists a root-vector ${\xi \in \Kset^n}$ of $f$ such that
\begin{equation} \label{eq:semilocal-error-estimate}
  \|x - \xi\| \preceq \alpha (E_f(x)) \, \|W_f(x)\| \quad\text{and}\quad
	\left \| \frac{x - \xi}{d(x)} \right \|_p \le h(E_f(x)),
\end{equation}
where the real functions $\alpha$ and $h$ are defined by
\begin{equation} \label{eq:function-alpha-h-definition}
\alpha(t) = 2 / (1 - (a - 1) t + \sqrt{(1 - (a - 1) t)^2 - 4 t}\,) \quad\text{and}\quad h(t) = t \, \alpha(t).
\end{equation}
Moreover, if the inequality \eqref{eq:semilocal-theorem-initial-condition-second-kind-Weierstrass} is strict,
then  the second inequality in \eqref{eq:semilocal-error-estimate} is strict too.
\end{thm}

In the sequel, we define the real function
\begin{equation}\label{eq:Omega}
\Omega(t) = \Psi(h(t))
\end{equation}
where $\Psi$ and $h$ are defined by \eqref{eq:Psi-Weierstrass} and \eqref{eq:function-alpha-h-definition}, respectively.

\medskip
Now we are in a position to state the main result of this paper.

%Theorem 3.2
\begin{thm} \label{thm:semilocal-theorem-Weierstrass-N}
Let ${f \in \Kset[z]}$ be a polynomial of degree ${n \ge 2}$, ${N \ge 1}$ and ${1 \le p \le \infty}$. Suppose ${x^{(0)} \in \Kset^n}$ is an initial guess with distinct components satisfying
\begin{equation}\label{eq:semilocal-theorem-Weierstrass-initial-condition}
E_f(x^{(0)}) \le 1 / (1 + \sqrt{a})^2 \quad\text{and}\quad \Omega(E_f(x^{(0)})) < 2,
\end{equation}
where the function ${E_f}$ is defined by \eqref{eq:FIC-third}, the function $\Omega$ is defined by \eqref{eq:Omega} and
${a = (n - 1) ^ {1/q}}$.
In the case ${n = 2}$ and ${p = \infty}$ we assume that the first inequality in
\eqref{eq:semilocal-theorem-Weierstrass-initial-condition} is strict.
Then $f$ has only simple zeros in $\Kset$ and
the Weierstrass-type iteration \eqref{eq:Weierstrass-high-order} is well-defined and converges to a root-vector $\xi$ of $f$
with order of convergence ${N + 1}$ and with error estimate
\begin{equation} \label{eq:semilocal-error-estimate-Weierstrass}
  \|x^{(k)} - \xi\| \preceq \alpha (E_f(x^{(k)})) \, \|W_f(x^{(k)})\|,
\end{equation}
for all ${k \ge 0}$ such that ${E_f(x^{(k)}) \le 1/ (1 + \sqrt{a})^2}$ and ${\Omega(E_f(x^{(k)})) < 2}$,
where the function $\alpha$ is defined by \eqref{eq:function-alpha-h-definition}.
\end{thm}

\begin{proof}
From the first inequality in \eqref{eq:semilocal-theorem-Weierstrass-initial-condition} and
Theorem~\ref{prop:semilocal-theorem-second-type-Weierstrass}, we conclude that
$f$ has only simple zeros and there exists a root-vector ${\xi \in \Kset^n}$ of $f$ such that
\[
\left \| \frac{x^{(0)} - \xi}{d(x^{(0)})} \right \|_p \le h(E_f(x^{(0)})).
\]
From this and the second inequality in \eqref{eq:semilocal-theorem-Weierstrass-initial-condition},
taking into account that $\Psi$ is increasing on ${[0,+\infty)}$, we obtain
\[
 \Psi\left(\left \| \frac{x^{(0)} - \xi}{d(x^{(0)})} \right \|_p  \right) \le \Psi (h(E_f(x^{(0)}))) =  \Omega(E_f(x^{(0)})) <  2.
\]
It follows from Theorem~\ref{thm:second-local-Weierstrass}
that the Weierstrass-type iteration \eqref{eq:Weierstrass-high-order} is well-defined and converges to $\xi$ with order of convergence
${N + 1}$. It remains to prove the estimate \eqref{eq:semilocal-error-estimate-Weierstrass}.
Suppose that for some ${k \ge 0}$,
\begin{equation}\label{eq:semilocal-theorem-Weierstrass-initial-condition-k}
E_f(x^{(k)}) \le 1 / (1 + \sqrt{a})^2 \quad\text{and}\quad \Omega(E_f(x^{(k)})) <  2.
\end{equation}
Then it follows from the first inequality in \eqref{eq:semilocal-theorem-Weierstrass-initial-condition-k} and
Theorem~\ref{prop:semilocal-theorem-second-type-Weierstrass} that there exists a root-vector
${\eta\in \Kset^n}$ of $f$ such that
\begin{equation} \label{eq:semilocal-error-estimate-k}
  \|x^{(k)} - \eta\| \preceq \alpha (E_f(x^{(k)})) \, \|W_f(x^{(k)})\| \quad\text{and}\quad
	\left \| \frac{x^{(k)} - \eta}{d(x^{(k)})} \right \|_p \le h(E_f(x^{(k)})).
\end{equation}
From the second inequality in \eqref{eq:semilocal-error-estimate-k} and the second inequality in
\eqref{eq:semilocal-theorem-Weierstrass-initial-condition-k}, we get
\[
 \Psi\left(\left \| \frac{x^{(k)} - \eta}{d(x^{(k)})} \right \|_p  \right) \le \Psi (h(E_f(x^{(k)}))) =  \Omega(E_f(x^{(k)})) <  2.
\]
By Theorem~\ref{thm:second-local-Weierstrass}, we conclude that the iteration
\eqref{eq:Weierstrass-high-order} converges to $\eta$. By the uniqueness of the limit, ${\eta = \xi}$.
Hence, the error estimate \eqref{eq:semilocal-error-estimate-Weierstrass} follows from the first inequality in
\eqref{eq:semilocal-error-estimate-k}.
\end{proof}

%Corollary 3.3
\begin{cor} \label{cor:semilocal-theorem-Weierstrass-first}
Let ${f \in \Kset[z]}$ be a polynomial of degree ${n \ge 2}$ and ${1 \le p \le \infty}$. Suppose ${x^{(0)} \in \Kset^n}$ is an initial guess with distinct components such that
\begin{equation}\label{eq:cor-semilocal-theorem-Weierstrass-initial-condition-third}
E_f(x^{(0)}) = \left \| \frac{W_f(x^{(0)})}{d(x^{(0)})} \right \|_p \le \frac{2}{5 a + 6} \,,
\end{equation}
where ${a = (n - 1) ^ {1/q}}$.
Then $f$ has only simple zeros in $\Kset$ and the Weierstrass-type iteration \eqref{eq:Weierstrass-high-order} is well-defined and converges to a root-vector of $f$ with order of convergence ${N + 1}$ and with error estimate
\eqref{eq:semilocal-error-estimate-Weierstrass}.
\end{cor}

\begin{proof}
The proof follows from Theorem~\ref{thm:semilocal-theorem-Weierstrass-N} because the initial condition
\eqref{eq:cor-semilocal-theorem-Weierstrass-initial-condition-third} implies
\eqref{eq:semilocal-theorem-Weierstrass-initial-condition}.
For simplicity, we set ${R = 2/(5a + 6)}$ and ${\mu = 1/(1 + \sqrt{a})^2}$.
To prove that $x^{(0)}$ satisfies \eqref{eq:semilocal-theorem-Weierstrass-initial-condition} it is sufficient to show that
${R \le \mu}$  and  ${\Omega(R) < 2}$, since $\Omega$ is increasing on $[0,+\infty]$. 
We prove only ${\Omega(R) < 2}$. 
We define the function ${G \colon \Rset_+ \to \Rset_+}$  
(the graph of $G$ is given in Fig.~\ref{fig:graphic-psitilde})
by:
\[
G(t) = (1 + 2 g(t)) e^{t g(t)}, \quad\text{where}\quad  g(t) =  4 / (3 t + 8 + \sqrt{9 t^2 + 8 t + 16}\,).
\]
%%%%%%%%%%%%%%%%%%%%%%%%%%%%%%%%%%%%%%%%%%%%%%%%%%
%%           Figure 1
%%%%%%%%%%%%%%%%%%%%%%%%%%%%%%%%%%%%%%%%%%%%%%%%%%
\begin{figure}[!ht]
    \centering
%  \framebox{\includegraphics[scale=1.0]{Figure1.pdf}}
 \includegraphics[scale=0.88]{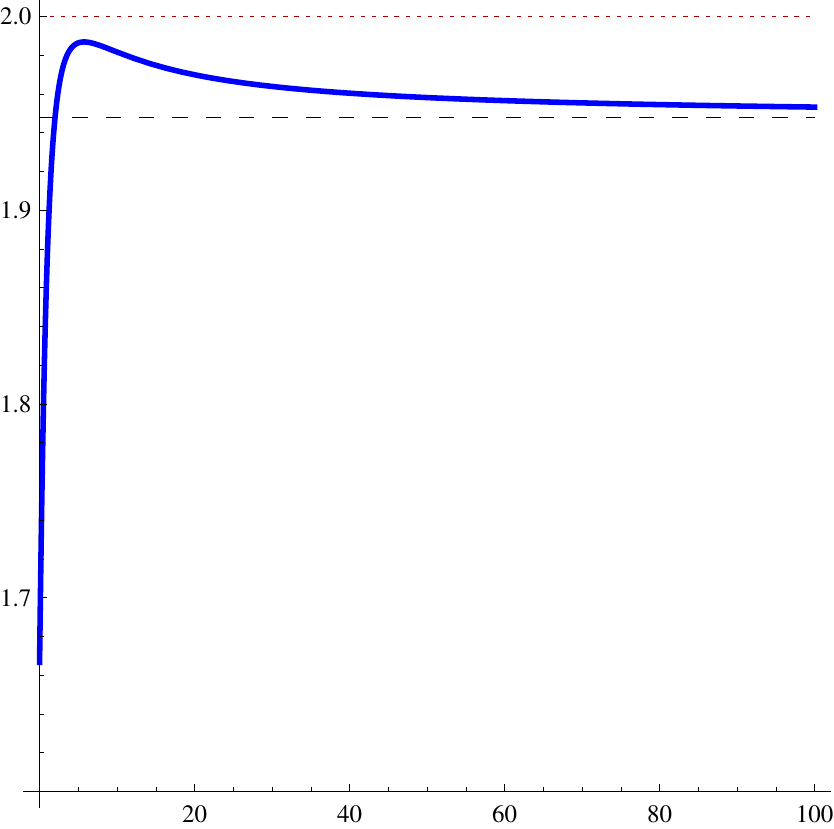}
\caption{Graph of the function $G$ on $\Rset_+$.}
    \label{fig:graphic-psitilde}
    \end{figure}
\FloatBarrier
\noindent
By using standard arguments of calculus, it is easy to prove that ${G(t) < 2}$ for all ${t \in \Rset_+}$ .
It follows from the definition of $\Omega$ and the well-known inequality ${\left( 1 + t / n \right)^n \le e^t}$ ${(t \in \Rset_+)}$ that
\begin{equation}
	\Omega(R) = (1 + 2 g(a)) \left(1 + \frac{a g(a)}{n - 1} \right)^{n - 1} \le G(a) < 2
\end{equation}
which completes the proof of the corollary.
\end{proof}

\subsection{Second semilocal convergence theorem}
\label{subsec:Second-semilocal-theorem-Weierstrass}

In this section we need another theorem of Proinov \cite{Pro15b} that enables us to transform local convergence theorems into semilocal ones.

%Theorem 3.4
\begin{thm}[Proinov \cite{Pro15b}] \label{prop:semilocal-theorem-second-kind}
Let ${f \in \Kset[z]}$ be a monic polynomial of degree ${n \ge 2}$. Suppose the initial guess ${x \in \Kset^n}$ with distinct components such that
\begin{equation}\label{eq:semilocal-theorem-initial-condition-second-kind}
   \left \| \frac{W_f(x)}{d(x)} \right \|_p \le \frac{R (1 - R)}{1 + (a - 1) R}
\end{equation}
for some ${1 \le p \le \infty}$ and ${0 < R \le 1/(1 + \sqrt{a})}$, where ${a = (n - 1) ^ {1/q}}$.
In the case ${n = 2}$ and ${p = \infty}$ we assume that inequality in
\eqref{eq:semilocal-theorem-initial-condition-second-kind} is strict.
Then $f$ has only simple zeros and there exists a root-vector ${\xi \in \Kset^n}$ of $f$ such that
\begin{equation} \label{eq:semilocal-error-estimate-initial-condition}
	\|x - \xi\| \preceq \alpha (E_f(x)) \, \|W(x)\| \quad\text{ and }\quad
	\left \| \frac{x - \xi}{d(x)} \right \|_p \le R,
\end{equation}	
where the real function $\alpha$ is defined as in \eqref{eq:function-alpha-h-definition}.
If the inequality \eqref{eq:semilocal-theorem-initial-condition-second-kind} is strict,
then the second inequality in \eqref{eq:semilocal-error-estimate-initial-condition} is strict too.
\end{thm}

%Theorem 3.5
\begin{thm} \label{thm:semilocal-theorem-Weierstrass}
Let ${f \in \Kset[z]}$ be a polynomial of degree ${n \ge 2}$, ${N \ge 1}$ and ${1 \le p \le \infty}$. Suppose ${x^{(0)} \in \Kset^n}$ is an initial guess with distinct components such that
\begin{equation} \label{eq:cor-semilocal-theorem-Weierstrass-initial-condition}
E_f(x^{(0)}) =  \left \| \frac{W_f(x^{(0)})}{d(x^{(0)})} \right \|_p < \mathcal{R} =
\frac{n (\sqrt[n]{2} - 1) (a + 2 - n (\sqrt[n]{2} - 1))}{(a + 2) ( a + 2 + n (a - 1)(\sqrt[n]{2} - 1))} \,,
\end{equation}
where ${a = (n - 1) ^ {1/q}}$.
Then $f$ has only simple zeros in $\Kset$ and the Weierstrass-type iteration \eqref{eq:Weierstrass-high-order} is well-defined and converges to a root-vector of $f$ with order ${N + 1}$ and with error estimate
\begin{equation} \label{eq:semilocal-error-estimate-Weierstrass-2}
  \|x^{(k)} - \xi\| \preceq \alpha (E_f(x^{(k)})) \, \|W_f(x^{(k)})\|,
\end{equation}
for all ${k \ge 0}$ such that ${E_f(x^{(k)}) \le \mathcal{R}}$,
where the function $\alpha$ is defined by \eqref{eq:function-alpha-h-definition}.
\end{thm}

\begin{proof}
Let ${R = (n(\sqrt[n]{2} - 1)) /(a + 2)}$. From the well-known inequality ${n (\sqrt[n]{2} - 1) < 1}$,
we obtain ${R <  1 / (1 + \sqrt{a})}$. On the other hand
\[
 \frac{R (1 - R)}{1 + (a - 1)R}  =
\frac{n (\sqrt[n]{2} - 1) (a + 2 - n (\sqrt[n]{2} - 1))}{(a + 2) ( a + 2 + n (a - 1)(\sqrt[n]{2} - 1))} \, .
\]
Therefore, \eqref{eq:cor-semilocal-theorem-Weierstrass-initial-condition} can be written in the form
\[
 \left \| \frac{W(x^{(0)})}{d(x^{(0)})} \right \|_p < \frac{R (1 - R)}{1 + (a - 1)R} \, .
\]
Then it follows from Theorem~\ref{prop:semilocal-theorem-second-kind} that $f$ has only simple zeros in $\Kset$ and
there exists a root-vector ${\xi \in \Kset^n}$ of $f$ such that
\[
\left \| \frac{x^{(0)} - \xi}{d(x^{(0)})} \right \|_p < R.
\]
Now Corollary~\ref{cor:second-local-Weierstrass} implies that the Weierstrass-type iteration \eqref{eq:Weierstrass-high-order} converges
to $\xi$ with order of convergence ${N + 1}$.
It remains to prove the error estimate \eqref{eq:semilocal-error-estimate-Weierstrass-2}.
Suppose that for some ${k \ge 0}$,
\begin{equation}\label{eq:semilocal-theorem-Weierstrass-corollary-initial-condition-k}
   E_f(x^{(k)}) = \left \| \frac{W_f(x^{(k)})}{d(x^{(k)})} \right \|_p <  \mathcal{R} = \frac{R (1 - R)}{1 + (a - 1) R} \, .
\end{equation}
Then it follows from Theorem~\ref{prop:semilocal-theorem-second-kind} that there exists a root-vector
${\eta\in \Kset^n}$ of $f$ such that
\begin{equation} \label{eq:semilocal-Weierstrass-error-estimate-k}
  \|x^{(k)} - \eta\| \preceq \alpha (E_f(x^{(k)})) \, \|W_f(x^{(k)})\| \quad\text{and}\quad
	\left \| \frac{x^{(k)} - \eta}{d(x^{(k)})} \right \|_p < R.
\end{equation}
From the second inequality in \eqref{eq:semilocal-Weierstrass-error-estimate-k} and Corollary~\ref{cor:second-local-Weierstrass}, we conclude that
the Weierstrass-type iteration \eqref{eq:Weierstrass-high-order} converges to $\eta$. By the uniqueness of the limit, we get ${\eta = \xi}$.
Consequently, the error estimate \eqref{eq:semilocal-error-estimate-Weierstrass-2} follows from the first inequality in
\eqref{eq:semilocal-Weierstrass-error-estimate-k}.
This ends the proof.
\end{proof}

%Corollary 3.6
\begin{cor} \label{cor:semilocal-theorem-Weierstrass-second}
Let ${f \in \Kset[z]}$ be a polynomial of degree ${n \ge 2}$, ${N \ge 1}$ and ${1 \le p \le \infty}$. Suppose ${x^{(0)} \in \Kset^n}$ is an initial guess with distinct components such that
\begin{equation}\label{eq:cor-semilocal-theorem-Weierstrass-initial-condition-second}
E_f(x^{(0)}) = \left \| \frac{W_f(x^{(0)})}{d(x^{(0)})} \right \|_p \le \frac{n (\sqrt[n]{2} - 1) (a + 1)}{(a + 2) (2 a + 1)} \,,
\end{equation}
where ${a = (n - 1) ^ {1/q}}$.
Then $f$ has only simple zeros in $\Kset$ and the Weierstrass-type iteration \eqref{eq:Weierstrass-high-order} is well-defined and converges with order ${N + 1}$ to a root-vector of $f$ with error estimate \eqref{eq:semilocal-error-estimate-Weierstrass-2}.
\end{cor}

\begin{proof}
It follows from Theorem~\ref{thm:semilocal-theorem-Weierstrass} and the inequality
${n (\sqrt[n]{2} - 1) < 1}$.
\end{proof}

%%%%%%%%%%%%%%%%%%%%%%%%%%%%%%%%%%%%%%%%%%%%%%%%%%
%%
%%              Numerical examples
%%
%%%%%%%%%%%%%%%%%%%%%%%%%%%%%%%%%%%%%%%%%%%%%%%%%%

%Section 4
\section{Numerical examples}
\label{sec:numerical-examples-Weierstrass}

In this section, we provide three numerical examples to show the applicability of Theorem~\ref{thm:semilocal-theorem-Weierstrass-N}.
We consider only the case ${p = \infty}$ since the results in other cases are similar.
Let ${f \in \Cset[z]}$ be a polynomial of degree ${n \ge 2}$ and let $x^{(0)} \in \Cset^n$ be an initial guess.
We consider the function of initial conditions ${E_f \colon \mathcal{D} \to \Rset_+}$  defined by
\begin{equation} \label{eq:FIC3-SM}
E_f(x) =  \left \| \frac{W_f(x)}{d(x)} \right \|_\infty \\ .
\end{equation}
Furthermore, we define the real function $\Omega$ as follows
\begin{equation} \label{eq:Psi}
\Omega(t) = (1 + 2  t \alpha(t))(1 + t  \alpha(t))^{n-1},
\end{equation}
where the function $\alpha$ is defined by
\begin{equation} \label{eq:alpha}
\alpha(t) = 2 / (1 - (n - 2) t + \sqrt{(1 - (n - 2) t)^2 - 4 t}\,).
\end{equation}
It follows from Theorem~\ref{thm:semilocal-theorem-Weierstrass-N} that if there exists an integer ${m \ge 0}$ such that
\begin{equation}  \label{eq:example-initial-conditions}
E_f(x^{(m)}) \le \mu = 1 / (1 + \sqrt{n - 1}\,)^2 \quad\text{and}\quad \Omega(E_f(x^{(m)})) < 2,
\end{equation}
then $f$ has only simple zeros and the Weierstrass-type iteration \eqref{eq:Weierstrass-high-order} starting from $x^{(0)}$ is well-defined and converges to a root-vector $\xi$ of $f$ with order of convergence ${N + 1}$.
Besides, the following a posteriori error estimate holds:
\begin{equation}  \label{eq:example-posteriori-estimate}
\|x^{(k)} - \xi\|_\infty \le \varepsilon_k,  \quad\text{where}\quad
\varepsilon_k = \alpha(E_f(x^{(k)})) \, \|W_f(x^{(k)})\|_\infty
\end{equation}
for all ${k \ge m}$ such that
\begin{equation}  \label{eq:example-a-posteriori-estimate}
 E_f(x^{(k)}) \le \mu \quad\text{and}\quad \Omega(E_f(x^{(k)})) < 2.
\end{equation}
In the examples below, we apply the Weierstrass-type methods \eqref{eq:Weierstrass-high-order} for some ${N \ge 1}$
using the stopping criterion
\begin{equation}  \label{eq:stop-criterion}
\varepsilon_k < 10^{-15} \qquad (k \ge m)
\end{equation}
together with \eqref{eq:example-a-posteriori-estimate}.
For given $N \ge 1$ we calculate the smallest $m \ge 0$ which satisfies the convergence condition \eqref{eq:example-initial-conditions},
the smallest $k \ge m$ for which the stopping criterion \eqref{eq:stop-criterion} is satisfied,
as well as the value of $\varepsilon_k$ for the last $k$. From these data it follows that:
1) $f$ has only simple zeros;
2) the $N$th Weierstrass-type iteration \eqref{eq:Weierstrass-high-order} starting from ${x^{(0)}}$ is well-defined and converges with order ${N + 1}$ to a root-vector of $f$;
3) at $k$th iteration the zeros of $f$ are calculated with an accuracy at least $\varepsilon_k$.

In Table~\ref{tab:HMSZ-example-Weierstrass-numerical-results-N=100} the values of iterations are given to 15 decimal places. The values of other quantities ($\mu$, $E_f(x^{(m)})$, etc.) are given to 6 decimal places.

%%%%%%%%%%%%%%%%%%%%%%%%%%%%%%%%%%%%%%%%%%%%%%%%%%
%%            Example 4.1  exmp:HMSZ
%%%%%%%%%%%%%%%%%%%%%%%%%%%%%%%%%%%%%%%%%%%%%%%%%%

%Example 4.1
\begin{exmp} \label{exmp:HMSZ}
Let us consider the polynomial ${f(z) = z^{3} - z}$ %%%%%, with roots $1$, ${-1}$, and $0$.
and the initial guess %${x^0 = (1.74,1.75,-3.49)}$
\[
{x^{(0)} = (1.74,1.75,-3.49)}
\]
which are taken from Hopkins et al. \cite{HMSZ94} and Niell \cite{Nie01}.
We have ${\mu = 0.171572}$.
For a given $N$, it can be seen from Table~\ref{tab:HMSZ-example} the value of $m$ which guarantees that the Weierstrass-type iteration
\eqref{eq:Weierstrass-high-order} starting from ${x^{(0)}}$ is well-defined and converges to a root-vector of $f$
with order of convergence ${N + 1}$,
the value of $k$ for which the stopping criterion \eqref{eq:stop-criterion} is satisfied and
the value of the error estimate $\varepsilon_k$ which is guaranteed from Theorem~\ref{thm:semilocal-theorem-Weierstrass-N}.
For example, for ${N = 100}$ at  the second iteration ${(m = 2)}$ we have proved that the method is convergent with order of convergence 101 and at the third iteration ${(k = 3)}$ we have calculated the zeros of $f$ with accuracy less than ${10^{-523}}$.
Moreover, at the fourth iteration we have obtained the zeros of $f$ with accuracy less than ${10^{-52900}}$.
The numerical results for ${N = 100}$ are shown in Table~\ref{tab:HMSZ-example-Weierstrass-numerical-results-N=100}.

%%%%%%%%%%%%%%%%%%%%%%%%%%%%%%%%%%%%%%%%%%%%%%%%%%
%%            Table 1
%%%%%%%%%%%%%%%%%%%%%%%%%%%%%%%%%%%%%%%%%%%%%%%%%%

\begin{table}[!ht]
  \centering
  \captionsetup{width=.955\textwidth}
  \caption{Values of $m$, $k$ and $\varepsilon_k$ for Example~\ref{exmp:HMSZ}.}
  \label{tab:HMSZ-example}
  \begin{tabular} {c c c c c|c l l}
\hline\\[-10pt]
$N$ & $m$ & ${E_f(x^{(m)})}$ & $\Omega(E_f(x^{(m)}))$ & $\varepsilon_m$ & $k$ & \multicolumn{1}{c}{$\varepsilon_k$} & \multicolumn{1}{c}{$\varepsilon_{k + 1}$}\\[+1.25pt]
\hline \\[-10pt]
1 & $12$ & $0.029714$ & $1.131702$ & $3.311488 \times 10^{-2}$  & $16$ & $5.496409 \times 10^{-26}$ & $3.000715 \times 10^{-51}$\\
2 & $6$  & $0.007688$ & $1.031545$ & $7.903736 \times 10^{-3}$  & $8$  & $2.463566 \times 10^{-21}$ & $7.688556 \times 10^{-63}$\\
3 & $6$  & $0.000216$ & $1.000867$ & $2.169611 \times 10^{-4}$  & $8$  & $1.692612 \times 10^{-59}$ & $8.138142 \times 10^{-236}$\\
4 & $4$  & $0.007479$ & $1.030664$ & $7.656408 \times 10^{-3}$  & $6$  & $2.712088 \times 10^{-66}$ & $1.252586 \times 10^{-330}$\\
5 & $6$  & $0.000000$ & $1.000000$ & $3.741978 \times 10^{-8}$  & $7$  & $1.837441 \times 10^{-45}$ & $3.058350 \times 10^{-269}$\\
6 & $4$  & $0.000361$ & $1.001445$ & $3.613767 \times 10^{-4}$  & $5$  & $7.021265 \times 10^{-29}$ & $1.900890 \times 10^{-199}$\\
7 & $3$  & $0.016712$ & $1.070710$ & $1.766014 \times 10^{-2}$  & $4$  & $5.881957 \times 10^{-17}$ & $1.306375 \times 10^{-131}$\\
8 & $4$  & $0.000000$ & $1.000000$ & $\,\,6.811047 \times 10^{-11}$
                                                                & $5$  & $1.439954 \times 10^{-95}$ & $1.144468 \times 10^{-857}$\\
9 & $3$  & $0.013852$ & $1.058033$ & $1.387643 \times 10^{-2}$  & $4$  & $2.122314 \times 10^{-19}$ & $1.503595 \times 10^{-187}$\\
10 & $4$ & $0.002015$ & $1.008114$ & $2.019382 \times 10^{-3}$  & $5$  & $1.020330 \times 10^{-36}$ & $2.321516 \times 10^{-402}$\\
100 &$2$ & $0.000006$ & $1.000026$ & $6.628377 \times 10^{-6}$  & $3$  & $2.609028 \times 10^{-524}$& $3.867338 \times 10^{-52901}$\\[+1.25pt]
\hline
\end{tabular}
\end{table}
\FloatBarrier

%%%%%%%%%%%%%%%%%%%%%%%%%%%%%%%%%%%%%%%%%%%%%%%%%%
%%            Table 2
%%%%%%%%%%%%%%%%%%%%%%%%%%%%%%%%%%%%%%%%%%%%%%%%%%

\begin{table}[!ht]
  \centering
  \captionsetup{width=.74\textwidth}
  \caption{Numerical results for Example~\ref{exmp:HMSZ} in the case $N = 100$.}
  \label{tab:HMSZ-example-Weierstrass-numerical-results-N=100}
  \begin{tabular} {l l l l}
\hline\\[-10pt]
$iteration$ & \multicolumn{1}{c}{$x_1^{(k)}$}  & \multicolumn{1}{c}{$x_2^{(k)}$} & \multicolumn{1}{c}{$x_3^{(k)}$}\\[+2.5pt]
\hline\\[-10pt]
$k = 0$ & $1.74             $ & $\;\; 1.75            $ & $-3.49             $\\
$k = 1$ & $1.149415748340902$ & $\;\; 1.975676419092484$ & $-2.359878141616537$ \\
$k = 2$ & $0.999998661360835$ & $   - 0.000006628312624$ & $-1.000004865683659$ \\
$k = 3$ & $1.000000000000000$ & $\;\; 0.000000000000000$ & $-1.000000000000000$ \\[+1.25pt]
\hline
\end{tabular}
\end{table}
\FloatBarrier

\end{exmp}

%%%%%%%%%%%%%%%%%%%%%%%%%%%%%%%%%%%%%%%%%%%%%%%%%%
%%            Example 4.2  exmp:SP
%%%%%%%%%%%%%%%%%%%%%%%%%%%%%%%%%%%%%%%%%%%%%%%%%%

%Example 4.2
\begin{exmp} \label{exmp:SP}
Consider the polynomial
\(
f(z) = z^{7} -  z^{5} - 10 z^{4} - z^{3} - z + 10
\)
and the initial guess
\[
x^{(0)} = (2.3 + 0.1 i, 1.2 + 0.2 i, -0.8 - 0.2 i, 0.1  + 1.3 i, - 0.2 - 0.8 i, -1.2 + 2.2 i, -1.2 - 1.8 i)
\]
which are taken from Sakurai and Petkovi\'c \cite{SP96}.
In this case ${\mu = 0.084040}$.
The results for this example are presented in Table~\ref{tab:SP-example}.
For example, we can see that for ${N = 4}$ at the first iteration ($m = 1$) we prove the convergence of the method, and at the third iteration ($k = 3$) we obtain the zeros of $f$ with accuracy less than ${10^{-58}}$.

%%%%%%%%%%%%%%%%%%%%%%%%%%%%%%%%%%%%%%%%%%%%%%%%%%
%%            Table 3
%%%%%%%%%%%%%%%%%%%%%%%%%%%%%%%%%%%%%%%%%%%%%%%%%%

\begin{table}[!ht]
  \centering
  \captionsetup{width=.92\textwidth}
  \caption{Values of $m$, $k$ and $\varepsilon_k$ for Example~\ref{exmp:SP}.}
  \label{tab:SP-example}
 \begin{tabular} {c c c c c|c l l}
\hline\\[-10pt]
$N$ & $m$ & ${E_f(x^{(m)})}$ & $\Omega(E_f(x^{(m)}))$ & $\varepsilon_m$ & $k$ & \multicolumn{1}{c}{$\varepsilon_k$} & \multicolumn{1}{c}{$\varepsilon_{k + 1}$}\\[+1.25pt]
\hline \\[-10pt]
1 & $2$ & $0.007526$ & $1.064790$ & $1.116392 \times 10^{-2}$ & $5$ & $1.796060 \times 10^{-17}$ & $2.792108 \times 10^{-34}$\\
2 & $1$ & $0.035532$ & $1.427605$ & $6.352229 \times 10^{-2}$ & $4$ & $1.209144 \times 10^{-39}$ & $7.010810 \times 10^{-118}$\\
3 & $1$ & $0.013767$ & $1.126494$ & $2.129981 \times 10^{-2}$ & $3$ & $2.368469 \times 10^{-31}$ & $4.291912 \times 10^{-123}$\\
4 & $1$ & $0.004823$ & $1.040419$ & $6.681020 \times 10^{-3}$ & $3$ & $1.000227 \times 10^{-59}$ & $8.418384 \times 10^{-297}$\\
5 & $1$ & $0.001903$ & $1.015502$ & $2.840694 \times 10^{-3}$ & $2$ & $2.619223 \times 10^{-17}$ & $7.631970 \times 10^{-101}$\\
6 & $1$ & $0.000695$ & $1.005604$ & $9.366066 \times 10^{-4}$ & $2$ & $1.157166 \times 10^{-22}$ & $9.947018 \times 10^{-156}$\\
7 & $1$ & $0.000253$ & $1.002031$ & $3.750097 \times 10^{-4}$ & $2$ & $1.968419 \times 10^{-29}$ & $1.470015 \times 10^{-231}$\\
8 & $1$ & $0.000107$ & $1.000862$ & $1.444295 \times 10^{-4}$ & $2$ & $3.245945 \times 10^{-36}$ & $1.808200 \times 10^{-322}$\\
9 & $1$ & $0.000038$ & $1.000306$ & $5.655465 \times 10^{-5}$ & $2$ & $9.224622 \times 10^{-45}$ & $3.800354 \times 10^{-443}$\\
10 & $1$& $0.000015$ & $1.000124$ & $2.091765 \times 10^{-5}$ & $2$ & $1.833150 \times 10^{-53}$ & $1.621635 \times 10^{-583}$\\
100& $1$& $0.000000$ & $1.000000$ & $\,\,1.325425 \times 10^{-40}$
                                                              &$1$ & $1.325425 \times 10^{-40}$ & $1.089487 \times 10^{-4036}$\\[+1.25pt]
\hline
\end{tabular}
\end{table}
\FloatBarrier

\end{exmp}

%%%%%%%%%%%%%%%%%%%%%%%%%%%%%%%%%%%%%%%%%%%%%%%%%%
%%          Example 4.3  exmp:Wang-Zhao
%%%%%%%%%%%%%%%%%%%%%%%%%%%%%%%%%%%%%%%%%%%%%%%%%%

%Example 4.3
\begin{exmp} \label{exmp:Wang-Zhao}
In 1997, Wang and Zhao \cite{WZ97} studied the convergence behavior of the Weierstrass method
\eqref{eq:Weierstrass-iteration} for the polynomials
\begin{equation} \label{eq:Wang-Zhao-polynomials}
f(z)=z^{20}-1 \quad\text{and}\quad f(z)=z^{30}-1 .
\end{equation}
In their paper they say: ``\emph{... when we intended to solve the equation ${z^{20}-1 = 0}$, we could not find any suitable initial values to use algorithm} \eqref{eq:Weierstrass-iteration}''. Actually, this is not true.
There are many initial approximations for which the Weierstrass method \eqref{eq:Weierstrass-iteration} is convergent for the polynomials
\eqref{eq:Wang-Zhao-polynomials}.
We have made two types of experiments.

First, we compute the zeros of the polynomials \eqref{eq:Wang-Zhao-polynomials} by using the Weierstrass algorithm
\eqref{eq:Weierstrass-iteration} with 1000 initial guesses ${x^{(0)} \in \Cset^n}$ such that ${\|x^{(0)}\|_\infty} \le 2$ given randomly.
By Theorem~\ref{thm:semilocal-theorem-Weierstrass-N} (${N = 1}$), we obtain that the Weierstrass method starting from each of these random initial approximations ${x^{(0)}}$ is well-defined and convergent to a root-vector of $f$.

Second, we compute the zeros of the polynomials \eqref{eq:Wang-Zhao-polynomials} by using the Weierstrass-type iterations
\eqref{eq:Weierstrass-high-order} for some
${N \ge 1}$ (including the case ${N = 1}$) with Aberth's initial approximation ${x^{(0)} \in \Cset^n}$ given by (see \cite{Abe73})
\begin{equation} \label{eq:Abert-initial-guess-R-Cauchy}
x_\nu^{(0)} = r_0 \exp{(i \theta_\nu)}, \quad \theta_\nu = \frac{\pi}{n} \left(2\nu - \frac{3}{2}\right), \quad \nu = 1,\ldots,n,
\end{equation}
where $n$ is the degree of the corresponding polynomial. We have made these experiments for ${r_0 = 1, 1.1, 1.2,\ldots,1.9,2}$.
Again, in all cases, we prove the convergence of the methods.
In this example, we present the results for the second type of experiments for ${r_0 = 2}$.
For the polynomial ${f(z) = z^{20} - 1}$ we have ${\mu = 0.034821}$ and the obtained results can be seen in
Table~\ref{tab:Wang-Zhao-example-Abert-power-20}.
For example, we can see that for ${N = 61}$ at the seven iteration we have calculated the zeros of $f$
with accuracy less than ${10^{-14153}}$.

%%%%%%%%%%%%%%%%%%%%%%%%%%%%%%%%%%%%%%%%%%%%%%%%%%
%%            Table 4
%%%%%%%%%%%%%%%%%%%%%%%%%%%%%%%%%%%%%%%%%%%%%%%%%%

\begin{table}[!ht]
  \centering
   \captionsetup{width=.95\textwidth}
  \caption{Values of $m$, $k$ and $\varepsilon_k$ for Example~\ref{exmp:Wang-Zhao} for ${f(z) = z^{20} - 1}$.}
  \label{tab:Wang-Zhao-example-Abert-power-20}
 \begin{tabular} {c c c c c|c l l}
\hline\\[-10pt]
$N$ & $m$ & ${E_f(x^{(m)})}$ & $\Omega(E_f(x^{(m)}))$ & $\varepsilon_m$ & $k$ & \multicolumn{1}{c}{$\varepsilon_k$} & \multicolumn{1}{c}{$\varepsilon_{k + 1}$}\\[+1.25pt]
\hline \\[-10pt]
1 & $16$ & $0.005454$ & $1.135937$ & $1.906753 \times 10^{-3}$ & $19$ & $5.251672 \times 10^{-16}$ & $2.620105 \times 10^{-30}$\\
2 & $10$ & $0.008641$ & $1.241514$ & $3.249990 \times 10^{-3}$ & $12$ & $6.054274 \times 10^{-16}$ & $2.002780 \times 10^{-44}$\\
3 & $8$  & $0.006432$ & $1.165842$ & $2.298445 \times 10^{-3}$ & $10$ & $3.924632 \times 10^{-29}$ & $2.034074 \times 10^{-111}$\\
4 & $7$  & $0.003429$ & $1.079931$ & $1.147442 \times 10^{-3}$ & $9$  & $1.568679 \times 10^{-51}$ & $7.736874 \times 10^{-251}$\\
5 & $7$  & $0.000000$ & $1.000000$ & $1.310563 \times 10^{-8}$ & $8$  & $3.920705 \times 10^{-43}$ & $2.810626 \times 10^{-250}$\\
6 & $6$  & $0.000465$ & $1.009907$ & $1.469386 \times 10^{-4}$ & $7$  & $1.026738 \times 10^{-21}$ & $8.842207 \times 10^{-142}$\\
7 & $6$  & $0.000000$ & $1.000006$ & $9.113539 \times 10^{-8}$ & $7$  & $3.323098 \times 10^{-50}$ & $1.038511 \times 10^{-389}$\\
8 & $5$  & $0.014073$ & $1.494951$ & $6.079699 \times 10^{-3}$ & $7$  & $2.518063 \times 10^{-112}$& $2.700157 \times 10^{-997}$\\
9 & $5$  & $0.001649$ & $1.036367$ & $5.324415 \times 10^{-4}$ & $6$  & $8.150179 \times 10^{-25}$ & $8.150497 \times 10^{-233}$\\
10& $5$  & $0.000075$ & $1.001583$ & $2.357206 \times 10^{-5}$ & $6$  & $7.347516 \times 10^{-42}$ & $2.017354 \times 10^{-443}$\\
61& $5$  & $0.000069$ & $1.001472$ & $2.192754 \times 10^{-5}$ & $6$  & $5.604020 \times 10^{-230}$& $1.117175 \times 10^{-14154}$\\
100& $3$ & $0.000000$ & $1.000000$ & $\,\,4.366726 \times 10^{-17}$
                                                               & $3$  & $4.366726 \times 10^{-17}$ & $2.679890 \times 10^{-1555}$\\
101& $3$ & $0.000000$ & $1.000000$ & $\,\,1.612383 \times 10^{-17}$
                                                               & $3$  & $1.612383 \times 10^{-17}$ & $8.163089 \times 10^{-1615}$\\[+1.25pt]
\hline
\end{tabular}
\end{table}
\FloatBarrier

In the Figure~\ref{fig:Wang-Zhao-example}, we present the trajectories of approximations generated by the methods
\eqref{eq:Weierstrass-high-order} for ${N = 1}$ after $19$ iterations and ${N = 61}$ after $6$ iterations.
The trajectories of approximations generated by other Weierstrass-type methods \eqref{eq:Weierstrass-high-order} are similar either to ${N = 1}$ or to ${N = 61}$.

%%%%%%%%%%%%%%%%%%%%%%%%%%%%%%%%%%%%%%%%%%%%%%%%%%
%%           Figure 2
%%%%%%%%%%%%%%%%%%%%%%%%%%%%%%%%%%%%%%%%%%%%%%%%%%

\begin{figure}[!ht]
\centering
%%\framebox{
\begin{minipage}[b]{0.45\linewidth}
\includegraphics[scale=0.65]{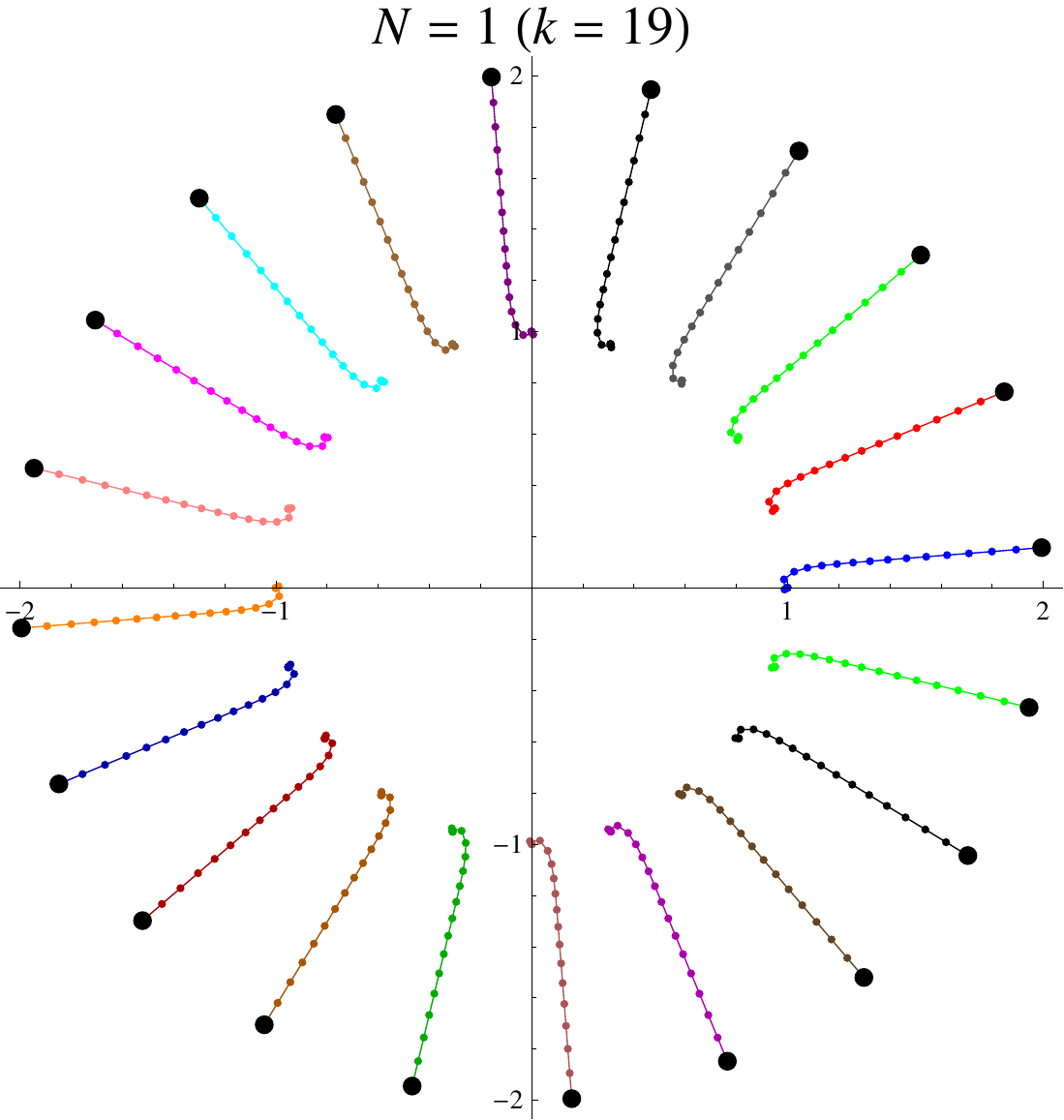}
\end{minipage}
\quad
\begin{minipage}[b]{0.45\linewidth}
\includegraphics[scale=0.65]{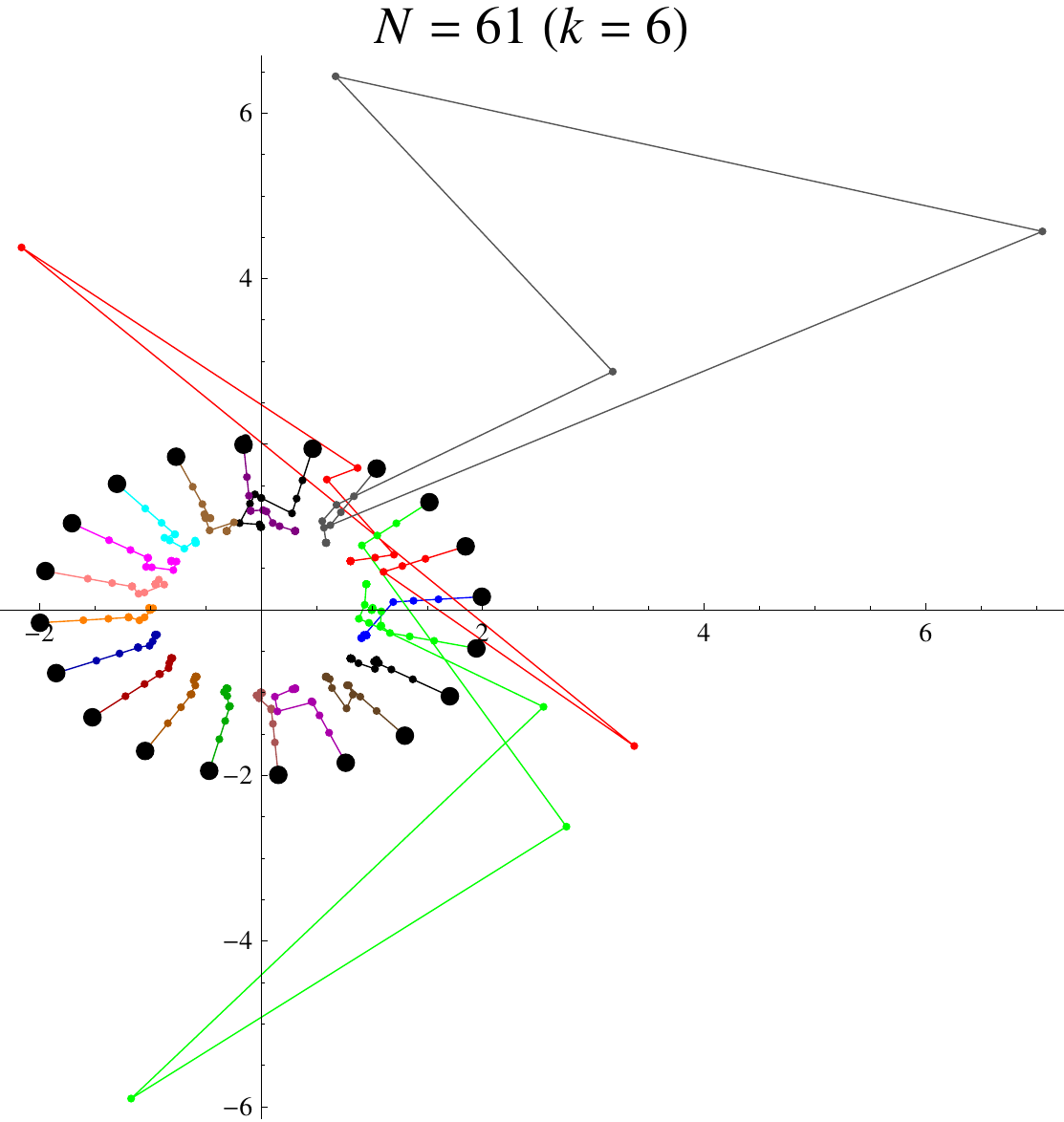}

\end{minipage}%}
\caption{Trajectories of approximations for the polynomial ${f(z) = z^{20} - 1}$.}
\label{fig:Wang-Zhao-example}
\end{figure}
\FloatBarrier

The situation is similar for the polynomial ${f(z) = z^{30} - 1}$. In this case ${\mu = 0.024527}$ and the obtained numerical results are presented in Table~\ref{tab:Wang-Zhao-example-Abert-power-30}.
For example, we can see that for ${N = 101}$ at the sixth iteration we have obtained the zeros of $f$
with accuracy less than ${10^{-95 810}}$.

%%%%%%%%%%%%%%%%%%%%%%%%%%%%%%%%%%%%%%%%%%%%%%%%%%
%%            Table 5
%%%%%%%%%%%%%%%%%%%%%%%%%%%%%%%%%%%%%%%%%%%%%%%%%%

\begin{table}[!ht]
  \centering
   \captionsetup{width=.95\textwidth}
  \caption{Values of $m$, $k$ and $\varepsilon_k$ for Example~\ref{exmp:Wang-Zhao} for ${f(z) = z^{30} - 1}$.}
  \label{tab:Wang-Zhao-example-Abert-power-30}
 \begin{tabular} {c c c c c|c l l}
\hline\\[-10pt]
$N$ & $m$ & ${E_f(x^{(m)})}$ & $\Omega(E_f(x^{(m)}))$ & $\varepsilon_m$ & $k$ & \multicolumn{1}{c}{$\varepsilon_k$} & \multicolumn{1}{c}{$\varepsilon_{k + 1}$}\\[+1.25pt]
\hline \\[-10pt]
1 & $23$ & $0.004903$ & $1.193434$ & $1.196341 \times 10^{-3}$ & $26$ & $1.664050 \times 10^{-16}$ & $4.015143 \times 10^{-31}$\\
2 & $15$ & $0.000303$ & $1.009546$ & $6.408814 \times 10^{-5}$ & $17$ & $3.307885 \times 10^{-29}$ & $7.610048 \times 10^{-84}$\\
3 & $12$ & $0.000132$ & $1.004131$ & $2.780420 \times 10^{-5}$ & $14$ & $3.153464 \times 10^{-56}$ & $3.014782 \times 10^{-219}$\\
4 & $10$ & $0.003933$ & $1.147300$ & $9.286229 \times 10^{-4}$ & $12$ & $5.264378 \times 10^{-50}$ & $1.787341 \times 10^{-242}$\\
5 & $9$  & $0.003966$ & $1.148751$ & $9.371733 \times 10^{-4}$ & $11$ & $4.726532 \times 10^{-71}$ & $7.146541 \times 10^{-417}$\\
6 & $9$  & $0.000000$ & $1.000000$ & $3.581766 \times 10^{-9}$ & $10$ & $7.028904 \times 10^{-53}$ & $7.878100 \times 10^{-359}$\\
7 & $8$  & $0.000337$ & $1.010607$ & $7.117318 \times 10^{-5}$ & $9$  & $8.244856 \times 10^{-26}$ & $2.877668 \times 10^{-193}$\\
8 & $8$  & $0.000000$ & $1.000000$ & $4.511178 \times 10^{-9}$ & $9$  & $1.511999 \times 10^{-66}$ & $8.070636 \times 10^{-584}$\\
9 & $7$  & $0.006793$ & $1.299041$ & $1.773403 \times 10^{-3}$ & $8$  & $1.009302 \times 10^{-18}$ & $3.108310 \times 10^{-170}$\\
10 &$7$  & $0.000195$ & $1.006113$ & $4.110752 \times 10^{-5}$ & $8$  & $2.188233 \times 10^{-37}$ & $2.263137 \times 10^{-392}$\\
101& $4$ & $0.000000$ & $1.000000$ & $\,\,4.263459 \times 10^{-11}$
                                                               & $5$  & $3.419093 \times 10^{-941}$ & $5.715983 \times 10^{-95811}$\\[+1.25pt]
\hline\\
\end{tabular}
\end{table}
\FloatBarrier

\end{exmp}

%Remark 4.4
\begin{rem}
Let us note that the numerical results of Examples \ref{exmp:HMSZ}, \ref{exmp:SP} and \ref{exmp:Wang-Zhao}
can also be obtained by each of
Theorem~\ref{thm:semilocal-theorem-Weierstrass},
Corollary~\ref{cor:semilocal-theorem-Weierstrass-first} and
Corollary~\ref{cor:semilocal-theorem-Weierstrass-second}.
\end{rem}

\section*{Acknowledgements}
This research is supported by Project NI15-FMI-004 of Plovdiv University.

%%%%%%%%%%%%%%%%%%%%%%%%%%%%%%%%%%%%%%%%%%%%%%%%%%
%%
%%              References
%%
%%%%%%%%%%%%%%%%%%%%%%%%%%%%%%%%%%%%%%%%%%%%%%%%%%

\end{document}